\documentclass[12pt,reqno]{amsart}
\usepackage[dvipdfmx]{graphicx, color} 
\usepackage{mathrsfs} 
\usepackage{amssymb, amsmath, amsthm, amscd}
\usepackage{tikz}
 
\theoremstyle{plain} 
\newtheorem{theorem}{Theorem}[section]
\newtheorem{lemma}[theorem]{Lemma}
\newtheorem{corollary}[theorem]{Corollary} 
\newtheorem{proposition}[theorem]{Proposition} 

\theoremstyle{definition} 
\newtheorem{definition}{Definition}[section]

\newtheorem{example}[definition]{Example}

\theoremstyle{remark} 
\newtheorem{remark}{Remark}[section] 

\newtheorem{problem}[remark]{Problem}

\numberwithin{equation}{section}

\begin{document}

%
%

\title{%
	On subgroups of Brin-Thompson groups  $nV$ 
}

\author[S. Kojima]{%
    Sadayoshi Kojima
}
\address{%
	Tokyo Institute of Technology, 
	Professor Emiritus 
} 
\email{%
        sadayosi@me.com
}
\author[X. Sheng]{%
	Xiaobing Sheng
} 
\address{%
	The University of Osaka  
}
\email{%
	sheng.xiaobing.yke@osaka-u.ac.jp
}
\thanks{%
} 

\begin{abstract} 
	We prove that the Brin-Thompson group  $nV$  is 
	torsion locally finite for  $ n \geq 1$  which is known 
	only when  $n = 1$, 
	and  $nV$  contains continuum many copies of
	the additive group of the rationals  $\mathbb{Q}$  for  $n \geq 2$  which is known 
	to be false for the  $n = 1$ case.  
\end{abstract}

\maketitle

%
%
\section{Introduction}\label{Sec:Introduction}   

Thompson's groups $F$, $T$ and $V$  
were first introduced by Richard J. Thompson in the 1960s 
and their unusual properties have made them a good source of counterexamples 
to many conjectures in group theory.
Various generalisations of these groups have been studied in more recent years, 
including the higher dimensional Thompson groups 
(known as the Brin-Thompson groups) 
defined by Brin \cite{Brin} in 2004. 
Roughly speaking, the Brin-Thompson group, denoted by  $nV,$ 
comprises self-homeomorphisms of the product $\mathcal{C}^{n}$ of $n$ copies 
of the Cantor set $\mathcal{C}$ 
which send dyadic block decompositions of  $\mathcal{C}^n$  to other ones 
piecewise affinely.  

This paper is motivated by Burillo, Cleary and R\"over's survey \cite{BCR}  
on the subgroups of Thompson's group $V$.  
Among others, 
they presented two results on  $V$,  
one is torsion locally finiteness proved by R\"over in  \cite{R, Rover}, 
and the other is the non-existence of roots of arbitrary large order, 
also previously shown by Higman in  \cite{Higman}.   
We study these properties for the higher dimensional version $nV$ 
from a more combinatorial perspective.  

Recall that a group is torsion locally finite if  
all of its finitely generated torsion subgroups are finite.  
The property of torsion groups being finite seems to be generic 
as Burnside proposed as a problem in 1902, 
but this was proved to be false. 
While this slightly different version of the property, torsion locally finite, 
holds for many groups including hyperbolic groups and 
some small cancellation groups.  
However, there exist exotic examples such as Burnside groups 
and the Grigorchuk groups  \cite{Grigorchuk}  
that do not have this property.

In this paper, 
we first prove that  $nV$  is not exceptional,
namely, 
 
\begin{theorem}\label{Thm:Main1} 
	$nV$  is torsion locally finite for  $n \geq 1$.  
\end{theorem} 

On the other hand,
we show that   

\begin{theorem}\label{Thm:Main2} 
        For $n \geq 2$,  
	then the Brin-Thompson group  $nV$  contains continuum many copies of
	the additive group of the rationals  $\mathbb{Q}$  sharing 
	the subgroup isomorphic to  $\mathbb{Z}$.  
\end{theorem} 

This result provides another concrete example 
to Problem \ref{Prob:MainProblem} proposed in Kourovka notebook \cite{K} 
by M. Bridson and P. de la Harpe 
(refer to the introduction of  \cite{BHM} for its short history)
apart from the central extension  $\bar{T}$  of  $T$  by  $\mathbb{Z}$ 
and a few of other examples that J. Belk, J. Hyde and  F. Matucci have constructed in \cite{BHM}.
\begin{problem}\label{Prob:MainProblem} 
	Find an explicit and ``natural" finitely presented group  $\Gamma$  
	and an embedding of the additive group of the rationals  $\mathbb{Q}$  in  
	$\Gamma$.  
\end{problem} 

\medskip  

\noindent
{\bf Organization of the paper:} 

For the rest of the paper, 
we define  
Brin-Thompson groups  $nV$  and 
provide a description for the combinatorial interpretation of the group elements 
in Section \ref{Sec:Preliminaries}. 
In Section \ref{Sec:Torsion}, 
we study torsion elements, 
torsion subgroups and we prove Theorem \ref{Thm:Main1}.
In Section \ref{Sec:Roots}, we focus on infinite order elements and prove Theorem \ref{Thm:Main2}. 
\medskip 

\noindent 
{\bf Acknowledgements:} 
Both authors would like to thank Corentin Bodart deeply for pointing out 
an error in the earlier version and for all the valuable suggestions which significantly improved the paper.
The second author wishes to thank Collin Bleak and Takuya Sakasai 
for useful discussions and supports in the early stage of this project. 
She would also like to thank Anthony Genevois and Nathan Corwin 
for sending her further references on this results while putting 
the manuscript  \cite{Sheng} on arXiv. 
Both authors were supported by JSPS KAKENHI Grant Number 21K03259 
and the second author is also supported by JSPS KAKENHI Grant Number 26K16982.

%
%
\medskip 
\section{Preliminaries}\label{Sec:Preliminaries}

\subsection{Dyadic block} 

The Brin-Thomson group  $nV$,  which we define later, 
is originally defined as a subgroup of the group of right continuous piecewise linear 
homeomorphisms of  $[0, 1]^n$.  
However, 
it is rather convenient to identify it as a subgroup of the group of 
homeomorphisms of the product of  $n$  copies of the Cantor set.  
Thus,  
let us recall the definition of the Cantor set.   

\begin{definition} 
	The Cantor set  $\mathcal{C}$   is defined to be a totally disconnected
	space obtained by iteratively deleting 
	the open middle thirds from the unit interval  $[0, 1]$  with subspace topology.  
\end{definition} 

The Cantor set  $\mathcal{C}$   can be topologically identified with the set of all binary strings  
$\{ 0, 1\}^{\mathbb{N}}$  starting from  $0$  with weak topology, 
where  $\mathbb{N}$  denotes the set of natural numbers.

\begin{definition}\label{Def:Subinterval} 
	A subinterval of  $[0, 1]$  of  
	the form  $\left[ \frac{\ell}{2^k}, \frac{\ell+1}{2^k}\right]$  
	where  $k, \ell \in \mathbb{N} \cup \{ 0 \}$  and  $\ell + 1 \leq 2^k$  
	is called a dyadic subinterval.  
\end{definition} 

If  $k = 0$  in Definition \ref{Def:Subinterval}, 
then  $\ell$  must be  $0$  and  $\left[ \frac{\ell}{2^k}, \frac{\ell+1}{2^k}\right] = [0, 1]$.    
When  $k \geq 1$, 
let  $\sum_{i=1}^{k} a_i 2^{-i} \in [0, 1]$    	
be a 2-adic expansion of  $\frac{\ell}{2^k}$.   
Then a dyadic subinterval  $\left[ \frac{\ell}{2^k}, \frac{\ell+1}{2^k}\right]$  
corresponds to 
an open and closed subset of  $\mathcal{C}$  
consisting of binary strings with the fixed prefix  $0a_1 \cdots a_k$.   
Notice that it is not homeomorphic to a dyadic subinterval of  $[0, 1]$  in fact, 
but it is a subset of  $\mathcal{C}$  homeomorphic to  $\mathcal{C}$.  

This identification leads us to the following 
definition of a dyadic subinterval in  $\mathcal{C}$  instead of  $[0, 1]$,  
which is more convenient for our later discussion.  

\begin{definition} 
	A dyadic subinterval of  $\mathcal{C}$  is defined to be 
	an open and closed subset of  $\mathcal{C}$  
	consisting of binary strings with a fixed prefix word  $ \alpha \in \{ 0, 1 \}^*$,  
	namely,  
	\begin{equation*} 
		\{ \alpha \omega \, ; \, \omega \in \{ 0, 1 \}^{\mathbb{N}} \}. 
	\end{equation*}   
\end{definition} 	 
	
Notice that any dyadic subinterval of  $\mathcal{C}$  is homeomorphic to  $\mathcal{C}$  
by rescaling. 
In particular, 
it is not a singleton and contains uncountably many elements.  

\begin{definition} 
	Let  $n$  be a positive integer.   
	A {\it dyadic subblock}  $B_* \subset \mathcal{C}^n $  is defined to be 
	a product of dyadic subintervals of  $\mathcal{C}$ such that
	\begin{enumerate}
	\item $B_*$  is a subset of  $\mathcal{C}^n$  
	\item there is a canonical affine homeomorphism 
	\begin{equation*} 
		\varphi_* : \mathcal{C}^n  \to B_* 
	\end{equation*}  
	by rescaling the coordinates (see Figure \ref{Fig:Rescale}). 
	\end{enumerate} 
	\begin{figure}[htp]  
		\begin{center} 	
	   \begin{tikzpicture}[baseline=-0.65ex, thick, scale=0.3]
                 \draw (0, 0) node{Rescaling map $\varphi_*$:};
                 \end{tikzpicture}\quad
		\begin{tikzpicture}[baseline=-0.65ex, thick, scale=0.3]
                 \draw (0, -5) to (0,5);
                 \draw (0, -5) to (10, -5);
                 \draw (10, -5) to (10,5);
                 \draw (10, 5) to (0, 5);
                 \filldraw [fill=lightgray,draw=black] (0,-5) -- (0,5) -- (10,5) -- (10,-5);
                 \end{tikzpicture}\quad
                 \begin{tikzpicture}[baseline=-0.65ex, thick, scale=0.3]
                 \draw (0, 0) node{$\mapsto$};
                 \end{tikzpicture}\quad	
        \begin{tikzpicture}[baseline=-0.65ex, thick, scale=0.3]
                 \draw (0, -5) to (0,5);
                 \draw (0, -5) to (10, -5);
                 \draw (10, -5) to (10,5);
                 \draw (10, 5) to (0, 5);
                 \draw (5, -5) to (5,5);
                 \draw (5, 0) to (10,0);
                 \draw (7.5, 0) to (7.5,5);
                 \draw (7.5,2.5) to (10,2.5);
                 \draw (8.75, 2.5) to (8.75,5);
                 \filldraw [fill=lightgray,draw=black] (5,0) -- (7.5,0) -- (7.5,5) -- (5,5);
               \end{tikzpicture}
			\caption{Rescaling map $\varphi_*$}
			\label{Fig:Rescale}  
		\end{center} 
	\end{figure} 
\end{definition} 

We now define the concept which plays a key role throughout this paper.  

\begin{definition}[Dyadic block]  
	A {\it dyadic block}  $B$  is defined to be a collection of subblocks
	\begin{equation*} 
		B = \{ B_i \} 
	\end{equation*} 
	such that 
	\begin{enumerate} 
	\item  
		each  $B_i$  is a subblock, 
	\item 
		$B_i \cap B_j = \emptyset$  if  $i \ne j$,   
	\item 
		$\bigcup_i B_i = \mathcal{C}^n $,  
	\item 
		$| B | < \infty$,  
	\end{enumerate} 
	where  $|B|$  is the cardinality of  $B$  which we call the length of  $B$, 
	namely, 
	$i \in \mathbb{N}$  runs between  $1$  and  $|B|$.
\end{definition} 

A dyadic subblock is a subset of  $\mathcal{C}^n $ 
whereas a dyadic block is not a subset of  $\mathcal{C}^n $.  
It is  a collection of finitely many dyadic subblocks 
and contains finitely many elements. 
We would now like to introduce an important relation between 
dyadic blocks.  

\begin{definition}[Refinement] 
	Suppose  $X = \{ X_i \}$  and  $Y = \{ Y_j \}$  are dyadic blocks. 
	$X$  is defined to be a {\it refinement} of  $Y$, 
	denoted by  $X \succeq Y$,  
	if for all  $i$, 
	there exists $j$  such that  $X_i \subseteq Y_j$ holds.    
\end{definition} 

\begin{lemma}\label{Lem:Intersection} 
	Suppose  $X = \{ X_i \}, \, Y = \{ Y_j \}$  are dyadic blocks. 
	Then  $X_i \cap Y_j$  is a dyadic subblock for any  $i$  and  $j$.  
	In particular, 
	a collection  $\{ X_i \cap Y_j \, ; \, i, j \}$  is a dyadic block.  
\end{lemma} 
\begin{proof} 
	This follows from the fact that 
	an intersection of dyadic subintervals is a dyadic subinterval.  
\end{proof} 

\begin{definition}[Common refinement] 
	The {\it common refinement} of  $X = \{ X_i \}$  and  $Y = \{ Y_j \}$  is defined as
	and denoted by
	\begin{equation*} 
		X \wedge Y = \{ X_i \cap X_j \, ; \, 1 \leq i \leq |X|, \, 1 \leq j \leq |Y| \},     
	\end{equation*} 
	which is a dyadic block by Lemma \ref{Lem:Intersection}.    
	\begin{figure}[htp]  
		\begin{center} 
		        \begin{tikzpicture}[baseline=-0.65ex, thick, scale=0.3]
                 \draw (0, -5) to (0,5);
                 \draw (0, -5) to (10, -5);
                 \draw (10, -5) to (10,5);
                 \draw (10, 5) to (0, 5);
                 \draw (5, -5) to (5,5);
                   \draw (5, -7) node{$X$};
               \end{tikzpicture}
                  \begin{tikzpicture}[baseline=-0.65ex, thick, scale=0.3]
                 \end{tikzpicture}\quad\quad	
        \begin{tikzpicture}[baseline=-0.65ex, thick, scale=0.3]
                 \draw (0, -5) to (0,5);
                 \draw (0, -5) to (10, -5);
                 \draw (10, -5) to (10,5);
                 \draw (10, 5) to (0, 5);
                 \draw (0, 0) to (10,0);
                                    \draw (5, -7) node{$Y$};
               \end{tikzpicture}\quad\quad	
                         \begin{tikzpicture}[baseline=-0.65ex, thick, scale=0.3]
                 \end{tikzpicture}\quad\quad	
        \begin{tikzpicture}[baseline=-0.65ex, thick, scale=0.3]
                 \draw (0, -5) to (0,5);
                 \draw (0, -5) to (10, -5);
                 \draw (10, -5) to (10,5);
                 \draw (10, 5) to (0, 5);
                 \draw (0, 0) to (10,0);
                                  \draw (5, -5) to (5,5);
                                                     \draw (5, -7) node{$X \wedge Y$};
               \end{tikzpicture}
			\caption{$X, Y \; and \; X \wedge Y$} 
			\label{Fig:CommonRefinement}  
		\end{center} 	
	\end{figure} 
\end{definition}  

\begin{remark} 
	The operation  $\wedge$  is commutative and associative.  
\end{remark} 	

\begin{remark}\label{Rm:CommonRefinement}  
	$X \wedge Y$  is a refinement of both  $X$  and  $Y$.  
\end{remark}

\medskip 
\subsection{Brin-Thompson group  $nV$} 

To define Brin-Thompson group  $nV$,  
we would like to introduce its combinatorial description for pragmatic computations.     
Let  $\mathcal{T}$  be the set of all triples  $(X, Y, \sigma)$, 
where  $X = \{ X_i \}$  and  $Y = \{ Y_j \}$  are dyadic blocks of  $\mathcal{C}^n$  
with the same length  $|X| = |Y| = m$,   
and  $\sigma$  is a permutation in 
the symmetric group $\mathfrak{S}_m = {\rm Sym} \{ 1, 2, \cdots, m \}$. 
The triple  $(X, Y, \sigma)$  induces a unique self-homeomorphism  $g$
of  $\mathcal{C}^n $ 
such that 
\begin{enumerate} 
\item 
	$g(X_i) = Y_{\sigma (i)}$, 
\item 
	If we identify  $X_i$  with the corresponding dyadic subblock in  $[0, 1]^n$, 
	$g|_{X_i}$  extends to an affine map of  $\mathbb{R}^n$  preserving coordinates with orientation.  
\end{enumerate} 
This correspondence between  $(X, Y, \sigma)$  and  $g$  leads us to define the following.  

\begin{definition} 
	Assigning  $g$  to  $(X, Y, \sigma)$,   	
	we define a map 
	\begin{equation*} 
		\bar{\pi} : \mathcal{T} \to {\rm Homeo} \, \mathcal{C}^n,   
	\end{equation*} 
	and denote the image  $\bar{\pi}(\mathcal{T})$  by  $nV < {\rm Homeo} \, \mathcal{C}^n$.  
	\begin{figure}[htp]  
		\begin{center}
	   \begin{tikzpicture}[baseline=-0.65ex, thick, scale=0.3]
                 \draw (0, 0) node{Induce homeomorphism $\vec{\pi}$:};
                 \end{tikzpicture}\quad
        \begin{tikzpicture}[baseline=-0.65ex, thick, scale=0.3]
                 \draw (0, -5) to (0,5);
                 \draw (0, -5) to (10, -5);
                 \draw (10, -5) to (10,5);
                 \draw (10, 5) to (0, 5);
                 \draw (5, -5) to (5,5);
                 \draw (2.5, 0) node{1};
                 \draw (7.5, 0) node{2};
               \end{tikzpicture}\quad	
                               \begin{tikzpicture}[baseline=-0.65ex, thick, scale=0.3]
                 \draw (5, 0) node{$\mapsto$};
                 \end{tikzpicture}\quad	
        \begin{tikzpicture}[baseline=-0.65ex, thick, scale=0.3]
                 \draw (0, -5) to (0,5);
                 \draw (0, -5) to (10, -5);
                 \draw (10, -5) to (10,5);
                 \draw (10, 5) to (0, 5);
                 \draw (0, 0) to (10,0);
                      \draw (5, 2.5) node{1};
                 \draw (5, -2.5) node{2};
               \end{tikzpicture}
			\caption{The map  $\bar{\pi}$} 
			\label{Fig:InducedHom}  
		\end{center} 	
	\end{figure} 
\end{definition} 

Obviously  $\bar{\pi}$  is not injective.  
For example, 
$id \in {\rm Homeo} \, \mathcal{C}^n $  can be represented by a triple  $(X, X, id)$  for any 
dyadic block  $X$  (see Figure \ref{Fig:identity}).  
\begin{figure}[htp]  
	\begin{center} 
	           \begin{tikzpicture}[baseline=-0.65ex, thick, scale=0.3]
                 \draw (0, -5) to (0,5);
                 \draw (0, -5) to (10, -5);
                 \draw (10, -5) to (10,5);
                 \draw (10, 5) to (0, 5);
                 \draw (5, -5) to (5,5);
                      \draw (2.5, 0) node{1};
                 \draw (7.5, 0) node{2};
               \end{tikzpicture}
                               \begin{tikzpicture}[baseline=-0.65ex, thick, scale=0.3]
                 \draw (5, 0) node{$\mapsto$};
                 \end{tikzpicture}
        \begin{tikzpicture}[baseline=-0.65ex, thick, scale=0.3]
                 \draw (0, -5) to (0,5);
                 \draw (0, -5) to (10, -5);
                 \draw (10, -5) to (10,5);
                 \draw (10, 5) to (0, 5);
                 \draw (5, -5) to (5,5);
                      \draw (2.5, 0) node{1};
                 \draw (7.5, 0) node{2};
               \end{tikzpicture}
               \quad
                     \begin{tikzpicture}[baseline=-0.65ex, thick, scale=0.3]
                 \draw (0, -5) to (0,5);
                 \draw (0, -5) to (10, -5);
                 \draw (10, -5) to (10,5);
                 \draw (10, 5) to (0, 5);
                 \draw (0, 0) to (10,0);
                     \draw (5, 2.5) node{1};
                 \draw (5, -2.5) node{2};
               \end{tikzpicture}
                               \begin{tikzpicture}[baseline=-0.65ex, thick, scale=0.3]
                 \draw (5, 0) node{$\mapsto$};
                 \end{tikzpicture}	
        \begin{tikzpicture}[baseline=-0.65ex, thick, scale=0.3]
                 \draw (0, -5) to (0,5);
                 \draw (0, -5) to (10, -5);
                 \draw (10, -5) to (10,5);
                 \draw (10, 5) to (0, 5);
                 \draw (0, 0) to (10,0);
                     \draw (5, 2.5) node{1};
                 \draw (5, -2.5) node{2};
               \end{tikzpicture}
		\caption{$(X, X, id) \; \; and \; \; (X', X', id)$} 
		\label{Fig:identity}  
	\end{center} 
\end{figure} 

We will give a group structure on an appropriate quotient of  $\mathcal{T}$ 
in order to say that  $\bar{\pi}$  induces  a homomorphism. 
To do this, 
we introduce a crucial terminology 
where we represent 
an element in  $nV$  by a pair of dyadic blocks with a permutation.

\begin{definition}[Admissibility]
	A dyadic block  $X$  is said to be 
	{\it admissible} for  $g \in nV$, 
	if  $g$  can be 
	represented by a triple  $(X, Y, \sigma)$  for some  $Y$  and  $\sigma$.  
\end{definition} 

If  $X = \{ X_i \}$  is admissible for  $g \in nV$, 
then  $Y = \{ g(X_i) \}$  is a dyadic block 
by the definition of the triple.  
Each  $g(X_i)$  is a dyadic subblock and  $Y$  is a collection of 
such subblocks that form a dyadic block,  
thus it is natural to denote  $Y$  by  $g(X)$.  
Under this notational convention, 
$g$  represents the map between dyadic blocks instead of  $\mathcal{C}^n$s,    
and we are abusing the notation  $g$  slightly.  
We distinguish them according to whether the source is  $\mathcal{C}^n$  or a dyadic block.  
When we use the notation  $g(X)$, 
$X$  must be admissible for  $g$.    

If  $X$  is admissible for  $g \in nV$, 
then  $(X, g(X), \sigma)$  now represents  $g$  for 
a suitable permutation  $\sigma$.  
This expression will be useful for notational simplification 
of our computation later on.  
Also, 
we say  $g$  acts on  $X$  if  $X$  is admissible for  $g$  for convenience 
though it is not a group action of  $nV$  on the set of dyadic blocks.   

The following three claims in the lemma will be frequently used in the later discussion 
without being referred to explicitly.  

\begin{lemma}\label{Lem:Admissibility}  
	Let  $X, Y$  be dyadic blocks of  $\mathcal{C}^n$.  
	\begin{enumerate} 
	\item 
		Suppose  $X \succeq Y$  and  $Y$  is admissible for  $g \in nV$.    
		Then,  
		$X$  is admissible for  $g$  and  $g(X) \succeq g(Y)$.  
	\item 
		If  $Y$  is admissible for  $g \in nV$,  
		then  $X \wedge Y$  is admissible for  $g$  for any dyadic block  $X$.
	\item 	
		If both  $X$  and  $Y$  are admissible for  $g \in nV$, 
		then  
		\begin{equation*} 
			g(X \wedge Y) = g(X) \wedge g(Y). 
		\end{equation*}   
	\end{enumerate} 
\end{lemma} 
\begin{proof} 
	For (1), 
	the admissibility of  $X$  for  $g$  follows immediately from the definition of 
	the admissibility of  $Y$  for  $g$, 
	and  $g(X)$  is a valid expression.  
	The relation  $g(X) \succeq g(Y)$  can be verified directly by the definition 
	of the action of  $g \in nV$  on dyadic blocks.   

	(2)  follows from Remark \ref{Rm:CommonRefinement} and  (1), 

	For (3),  
	since  $g(X \wedge Y) \succeq  g(X)$  and 
	$g(X \wedge Y) \succeq  g(Y)$  by Remark \ref{Rm:CommonRefinement} and (1),  
	$g(X \wedge Y) \succeq g(X) \wedge g(Y)$.  
	On the other hand, 
	since  $g(X_i) \cap g(Y_j) \supset g(X_i \cap Y_j)$, 
	$g(X) \wedge g(Y) \succeq g(X \wedge Y)$  
	and we obtain the desired equality.   
\end{proof} 

\begin{remark} 
	The role of admissibility of  $Y$  in Lemma \ref{Lem:Admissibility} (1)  is crucial.  
	If  $X \succeq Y$  and  $X$  instead of  $Y$  is admissible for  $g \in nV$,    
	$Y$  may not be admissible for  $g$.  	
	For instance, 
	if  $g$  is nontrivial, 
	$X$  is admissible for  $g$  and  $Y = \{ \mathcal{C}^n \}$  is a dyadic block with a single element,   
	then though  $X \succeq Y$  and  $g(X)$  does make sense by definition, but	  
	$Y$  is not admissible for  $g$.  
\end{remark} 

We now introduce a multiplication on  $\mathcal{T}$ which 
turned out to induce the group operation in its images $nV$.

\begin{definition} 
	Suppose we have two triples  
	$(X, Y, \sigma), \, (X', Y', \sigma') \in \mathcal{T}$  such that  
	$\bar{\pi} (X, Y, \sigma) = g$  and  $\bar{\pi} (X', Y', \sigma') = g'$. 
	Define the multiplication  $\cdot$  on  $\mathcal{T}$
	by 
	\begin{equation}\label{Eq:Product} 
		  (X, Y, \sigma) \cdot (X', Y', \sigma') 
		  = (g^{-1} (Y \wedge X'), g'(Y \wedge X'), \tau),   
	\end{equation} 
	where  $\tau$  is a permutation of  $|Y \wedge X'|$  letters 
	induced from  $\sigma$  and  $\sigma'$.  
	More precisely, $\tau = \sigma' \circ \sigma$.
\end{definition} 

The first two terms  
appeared in the right hand side of  (\ref{Eq:Product})  are valid expressions  
since  $Y$  is admissible for  $g^{-1}$  and  $X'$  is admissible for  $g'$.  
In particular, 
the expression of each dyadic block appeared in the product of two dyadic block pairs 
is automatically valid.

\begin{remark}  
	It should be remarked that  
	\begin{equation*} 
		\bar{\pi}((X, Y, \sigma) \cdot (X', Y', \sigma')) = h \circ g,  
	\end{equation*}  
	where  $\circ$  indicates the composition of the maps in  ${\rm Homeo} \, \mathcal{C}^n$.  
	Thus,  
	$\bar{\pi}$  becomes an anti-homomorphism between  
	$\mathcal{T}$  and  $nV < {\rm Homeo} \, \mathcal{C}^n$.
\end{remark} 

\begin{lemma}\label{associativity} 
	The multiplication  $\cdot$  on  $\mathcal{T}$  is associative and 
	the pair  $(\mathcal{T} \, , \, \cdot)$  forms a 
	monoid.  
\end{lemma} 
\begin{proof} 
	To show the associativity, 
	let  $(X, Y, \sigma), \, (X', Y', \sigma'), \, (X'', Y'', \sigma'') \in \mathcal{T}$ 
	define  $g, g', g'' \in nV$.  
	Then, 
	since permutations are defined by  $g, g', g''$  automatically, 
	we omit them from the notation for computation, 
	and we have 
	\begin{align*} 
		((X, Y) \cdot (X', Y')) \cdot (X'', Y'') 
		& = (g^{-1}(Y \wedge X'), g'(Y \wedge X')) \cdot (X'', Y'') \\
		& = ((gg')^{-1} (g'(Y \wedge X') \wedge X''), g''(g'(Y \wedge X') \wedge X'')),   
	\end{align*} 
	and 
	\begin{align*} 
		(X, Y) \cdot ((X', Y') \cdot (X'', Y'')) 
		& = (X, Y) \cdot (g'^{-1}(Y' \wedge X''), g''(Y' \wedge X'')) \\
		& = (g^{-1} (Y \wedge g'^{-1}(Y' \wedge X''), (g' g'')(Y \wedge (g'^{-1}(Y' \wedge X'')))). 
	\end{align*} 
	Both can be verified to be the same by using the following,
	\begin{equation*} 
		g(X \wedge Y) = \{ g(X_i) \cap g(Y_j) \, ; \, 1 \leq i \leq |X|,  1 \leq j \leq |Y| \}.     
	\end{equation*}  
	Moreover,
	$(\{ \mathcal{C}^n \}, \{ \mathcal{C}^n \}, id)$  is the identity element.   
\end{proof} 

We now introduce an equivalence relation  $\sim$  on  $\mathcal{T}$  to 
make it a group.

\begin{definition} 
	Suppose  $\bar{\pi}((X, Y, \sigma)) = g$  and  $\bar{\pi} (X', Y', \sigma') = g'$.  
	The equivalence relation  $\sim$  on  $\mathcal{T}$  is defined to be 
	\begin{equation*} 
		(X, Y, \sigma) \sim (X', Y', \sigma') 
		\, \Longleftrightarrow \,  
		g^{-1} g' = id.  
	\end{equation*} 
\end{definition} 

\begin{remark} 
	It is easy to check that  $\sim$  is an equivalence relation.  
	Moreover, 
	if  $(X, Y, \sigma) \sim (X', Y', \sigma')$, 
	then  
	\begin{equation*} 
		(X, Y, \sigma) \cdot (Y', X', \sigma^{-1}) 
		= (g^{-1}(Y \wedge Y'), (g')^{-1}(Y \wedge Y'), id),   
	\end{equation*} 
	where the source and target dyadic blocks on the right hand side are 
	the same.  
\end{remark}

\begin{lemma} 
	$(\mathcal{T}/\sim, \, \cdot)$  forms a group.  
\end{lemma} 
\begin{proof} 
	All we need is to find an inverse of  $(X, Y, \sigma)$, 
	that is  $(Y, X, \sigma^{-1})$.  
\end{proof}

\begin{proposition} 
	$\bar{\pi}$  induces a group anti-isomorphism  
	$\pi : \mathcal{T}/\sim \, \to {\rm Homeo} \, \mathcal{C}^n$  
	onto its image.   
	In particular, 
	$nV = \pi(\mathcal{T})$  is a group.  
\end{proposition} 
\begin{proof} 
	Suppose  $(X, Y, \sigma), \,  (X', Y', \sigma') \in \mathcal{T}$,   
	and let  $\pi(X, Y, \sigma) = g$  and  $\pi(X', Y', \sigma') = g'$.      
	Then, 
	\begin{equation*} 
		\pi(X, Y, \sigma) \circ \pi(X',Y', \sigma') = g \circ g'.  
	\end{equation*} 
	by definition.  
	On the other hand, 
	\begin{equation*}
	        (X, Y, \sigma_g) \cdot  (X', Y', \sigma_{g'}) 
	        = (g^{-1}(Y \wedge X'), g'(Y \wedge X'), \sigma_{gg'})  
	\end{equation*} 
	and 
	$g^{-1}(Y \wedge X') $ is taken to $g'(Y \wedge X')$  by 
	a map first  $g$  and then  $g'$, 
	thus 
	\begin{equation*} 
		\pi((X, Y, \sigma_g) \cdot  (X', Y', \sigma_{g'})) = g' \circ g.  
	\end{equation*}  
	This shows that  $\pi$  is a group anti-homomorphism.  
	Moreover, 
	the kernel of  $\pi$  consists of a single element  $[(X, X, id)] \in \mathcal{T}/\sim$, 
	$\pi$  is injective.  
\end{proof} 

\begin{definition}[Brin-Thompson group]
	$nV$  is called the $n$-dimensional Thompson group 
	or Brin-Thompson group.  
	The original Thompson group  $V$  is the  $n = 1$ version.  
	Due to literature, 
	we define the multiplication of  $g$  and  $h$  in  $nV$  by 
	\begin{equation*} 
		gh = h \circ g 
	\end{equation*} 
	where the right hand side is regarded as a composition  
	in  ${\rm Homeo} \, \mathcal{C}^n$.  
\end{definition} 

\begin{remark} 
	$nV$  is known to be finitely generated by  \cite{Brin}, 
	and hence admits a word metric together with its quasi-isometry class.  
\end{remark} 

Although  $X$  is admissible for  $gh$,  
$X$  may not be admissible for  $g$.  
For example, 
if  $h = g^{-1}$,  
then any  $X$  is admissible for  $gh$,  
however, not all dyadic blocks  $X$  are 
admissible for nontrivial  $g$.   

If  $X$  is admissible for  $g$  and  $g(X)$  is admissible for  $h$, 
then  
\begin{equation}\label{Eq:CompositionDB} 
	(gh)(X) = h(g(X)). 
\end{equation}   
Otherwise, 
the admissibility of  $X$  for  $gh$  does not immediately imply 
this identity, 
and we should keep in mind the difference between our notation and 
a usual one for the group action.  
The next proposition gives a useful criterion for   
(\ref{Eq:CompositionDB})  to be held.  

\begin{proposition}\label{Prop:CompositionDB}  
	If  $X$  is admissible for both  $g$  and  $gh$, 
	then  $g(X)$  is admissible for  $h$,  
	and in particular, 
	the identity  {\em (\ref{Eq:CompositionDB})}  holds. 
\end{proposition}
\begin{proof} 
	Let  $X = \{ X_i \}$  and  $g(X) = Y = \{ Y_j \}$.  
	Then, 
	for each  $i$, 
	there exists  $j$  such that  $g(X_i) = Y_j$.  
	On the other hand, 
	if we let  $(gh)(X) = Z = \{Z_k \}$, 
	then 
	$(gh)(X_i) = h(g(X_i)) = h(Y_j)$  as a subset of $\mathcal{C}^n$.   
	Thus for each  $j$,  
	there exists  $k$  such that  
	$h(Y_j) = Z_k$,     
	and therefore  $g(X)$  is admissible for  $h$.     
\end{proof}

%
%
\medskip 
\section{Torsion}\label{Sec:Torsion} 

In this section, 
we study the torsion elements in $nV$  and prove Theorem \ref{Thm:Main1} in the introduction 
which claims the torsion locally finiteness of  $nV$.

\subsection{Power of elements} 

We start with the diagrammatic structure of the power of an element in  $nV$.  

\begin{lemma} 
	Let  $g \in nV$  be an element represented by a triple  $(X, g(X), \sigma)$.  
	Then, 
	the target dyadic block of  $(X, g(X), \sigma)^i$  for  $i \in \mathbb{N}$  is 
	\begin{equation*} 
		\overbrace{g(g(g( \cdots (g(g}^{i}(X) \wedge X) \wedge X)\wedge \cdots ) \wedge X) \wedge X).  
	\end{equation*} 
\end{lemma} 
\begin{proof} 
	When  $i = 2$, 
	the common refinement  $g(X) \wedge X$  
	to define  $(X, g(X), \sigma)^2$  is admissible for  $g$  since  
	$X$  is admissible for  $g$, 
	and the target dyadic block of  $(X, g(X), \sigma)^2$  is  $g(g(X) \wedge X)$.  
	The general formula can be verified by induction.  
\end{proof} 

\begin{lemma}\label{Lem:Power}  
	Let  $X$  be a dyadic block of length  $m$.  
	Then an element  $g \in nV$  represented by 
	an identical block pair  $(X, X, \sigma) \, \in \mathcal{T}$  is torsion 
	for any permutation  $\sigma$  in 
	$\mathfrak{S}_m = {\rm Aut} \{ 1, 2, \cdots, m \}$.  
\end{lemma}  
\begin{figure}[htp]  
	\begin{center} 
	\begin{tikzpicture}[baseline=-0.65ex, thick, scale=0.3]
                 \draw (0, -5) to (0,5);
                 \draw (0, -5) to (10, -5);
                 \draw (10, -5) to (10,5);
                 \draw (10, 5) to (0, 5);
                 \draw (5, -5) to (5,5);
                 \draw (5, 0) to (10,0);
                 \draw (2.5, 0) node{1};
                 \draw (7.5, 2.5) node{2};
                 \draw (7.5, -2.5) node{3};
        \end{tikzpicture}\quad
        \begin{tikzpicture}[baseline=-0.65ex, thick, scale=0.3]
                 \draw (5, 0) node{$\mapsto$};
        \end{tikzpicture}\quad
        \begin{tikzpicture}[baseline=-0.65ex, thick, scale=0.3]
                 \draw (0, -5) to (0,5);
                 \draw (0, -5) to (10, -5);
                 \draw (10, -5) to (10,5);
                 \draw (10, 5) to (0, 5);
                 \draw (5, -5) to (5,5);
                 \draw (5, 0) to (10,0);
                 \draw (2.5, 0) node{3};
                 \draw (7.5, 2.5) node{1};
                 \draw (7.5, -2.5) node{2};
        \end{tikzpicture}
		\caption{An example of  $(X, X, \sigma)$} 
		\label{Fig:Power}  
	\end{center} 
\end{figure} 
\begin{proof} 
	This follows from the fact that  $g(X) = X$  and  $\sigma$  is of finite order.   
\end{proof} 

The following proposition is motivated by 
a suggestion from Collin Bleak through private communication with the second author.  

\begin{proposition}\label{Prop:IdenticalBrock} 
 	If $g \in nV$  is torsion of order  $p$,
	then there exists a dyadic block  $B$ 
	such that  $g(B) = B$. 
\end{proposition} 
\begin{proof} 
	Assume that a dyadic block  $X$  is admissible for  $g$.  
	Then,  
	the triple  $(X, g(X), \sigma)$  represents  $g$  and 
	referring to Lemma \ref{Lem:Power}, 
	we let  $B$  be  
	\begin{align*}
		B & = 
		\overbrace{g(g(g( \cdots (g(g}^{p}
		(X) \wedge X) \wedge X)\wedge \cdots ) \wedge X) \wedge X) \\
		& = \{ X_{i_1} \cap g(X_{i_2}) \cap \cdots \cap 
		                      g^{p-1}(X_{i_p})   \, ; \, 
			1 \leq i_1, i_2, \cdots, i_p \leq |X| \}.  
	\end{align*} 
	The last representation immediately implies  $g(B) = B$  since  $g^p = id$.  
\end{proof} 

For the power of a product of two torsion elements, 
we introduce some notations.  
Assume that  $g, h \in nV$  are torsion elements.  
Then, 
both of them can be represented by some identical block pairs  
by Proposition \ref{Prop:IdenticalBrock}.  
Namely, 
there exist dyadic blocks  $X$  and  $Y$  satisfying 
$g(X) = X$  and  $h(Y) = Y$.  
The product  $gh$  can be represented by a block pair 
\begin{equation*} 
	(X, X) \cdot (Y, Y) 
	= (g^{-1} (X \wedge Y), h(X \wedge Y)), 
\end{equation*} 
and if we let 
\begin{equation*} 
	(g^{-1} (X \wedge Y), h(X \wedge Y))^i
	= (S^i,T^i),    
\end{equation*} 
for  $i \geq 1$, 
then by the definition of the multiplication on  $\mathcal{T}$,  
\begin{align*} 
	(S^{i}, T^{i}) & = (S^{i-1}, T^{i-1}) \cdot (g^{-1}(X \wedge Y), h(X \wedge Y)) \\
	& = ((gh)^{1-i}(T^{i-1} \wedge g^{-1}(X \wedge Y)), (gh)(T^{i-1} \wedge g^{-1}(X \wedge Y)))  
\end{align*} 
for  $i \geq 2$.   
$S$  simply stands for a source block and  $T$ for a target one.  
	
\begin{remark}\label{Rm:Tadmissible}  
	Since  $T^{i-1} \wedge g^{-1}(X \wedge Y)$  for  $i \geq 2$  is 
	admissible for  $g$,  
	by Proposition \ref{Prop:CompositionDB}, 
	we have 
	\begin{equation*} 
		T^i
		= \overbrace{h(g( \cdots (h(g}^{2i}(h(X \wedge Y) \wedge g^{-1} (X \wedge Y))) \wedge \cdots ) 
		\wedge g^{-1} (X \wedge Y)) \wedge g^{-1} (X \wedge Y)).  
	\end{equation*} 
	Hence, in particular,  
	$T^i$  is admissible for  $h^{-1}$.  
\end{remark} 
	
Replacing  $gh$  by  $hg$,   
we can represent  $(hg)^i$  by  
\begin{equation*} 
	((Y, Y) \cdot (X, X))^i = (h^{-1}(Y \wedge X), g(Y \wedge X))^i = (D^i, R^i).  
\end{equation*}  
Here similarly,
$D$  simply stands for a domain block and  $R$  a range one.  

\begin{remark}\label{Rm:Radmissible} 
	Likewise,  
	$R^i$  is admissible for  $g^{-1}$.  
\end{remark} 

The following three lemmas will be established under 
the assumption that only when  $g$  and  $h$  are torsion.   

\begin{lemma}\label{Lem:Filtration} 
	$T^{i+1} \succeq T^{i}$  and  $R^{i+1} \succeq R^{i}$   for all  $i \geq 1$.   
\end{lemma} 
\begin{proof}  
	By the associativity of the product on  $\mathcal{T}$, 
	\begin{equation*} 
		(S^{i+1}, T^{i+1}) = (S^1, T^1) \cdot (S^i, T^i).   
	\end{equation*} 
	Thus, 
	\begin{equation*} 
		T^{i+1} = (gh)^i(T^1 \wedge S^i) \succeq (gh)^i(S^i) = T^i.    
	\end{equation*} 
	The same computation works to prove  $R^{i+1} \succeq R^i$.  
\end{proof} 

\begin{lemma}\label{Lem:FiltrationII} 
	$S^{i+1} \succeq S^{i}$  and  $D^{i+1} \succeq D^{i}$   for all  $i \geq 1$.   
\end{lemma} 
\begin{proof}  
	By the associativity of the product on  $\mathcal{T}$, 
	\begin{equation*} 
		(S^{i+1}, T^{i+1}) = (S^i, T^i) \cdot (S^1, T^1).   
	\end{equation*} 
	Thus, 
	\begin{equation*} 
		S^{i+1} = (gh)^{-i}(T^i \wedge S^1) \succeq (gh)^{-i}(T^i) = S^i.   
	\end{equation*} 
	The same computation works for  $D^{i+1} \succeq D^i$.  
\end{proof} 

\begin{lemma}\label{Lem:TandR} 	
        With the notations as above, 
        we have 
	\begin{equation*}
		h^{-1}(T^{i+1}) \succeq R^i  
		\quad \text{and} \quad 
		g^{-1}(R^{i+1}) \succeq T^i  
	\end{equation*} 
	for any  $i \geq 1$.  
\end{lemma} 

\begin{proof} 
	To see the first claim, 
        we prove a stronger version that  
        \begin{equation*} 
       		h^{-1}(T^{i+1}) \succeq R^{i} \wedge Y 
	\end{equation*} 
        by induction on $i$.  
        
        Since  $g^{-1}(X \wedge Y) \succeq g^{-1}(X) = X$  and  
        $h(X \wedge Y) \succeq h(Y) = Y$,  
        $g^{-1}(X \wedge Y) \wedge h(X \wedge Y) \succeq X \wedge Y$.  
        Also, 
        $X \wedge Y$  is admissible for  $gh$  and  $g$.  
        In particular, 
        \begin{align*} 
	        T^2 
	        & = (gh)(g^{-1} (X \wedge Y) \wedge h(X \wedge Y)) \\ 
	        & = h(g(g^{-1} (X \wedge Y) \wedge h(X \wedge Y)) \\ 
	        & \succeq h(g(X \wedge Y)), 
        \end{align*} 
        and hence  $h^{-1}(T^2) \succeq g(X \wedge Y) = R^1$.  
        Also, 
        since  
        \begin{equation*} 
		T^2 = h(g(g^{-1} (X \wedge Y) \wedge h(X \wedge Y))) 
		\succeq h(g(g^{-1}(X \wedge Y))) 
		\succeq h(X \wedge Y) 
		\succeq Y,   
        \end{equation*} 
        we have 
        \begin{equation}\label{Eq:Y} 
	        h^{-1}(T^2) \succeq  g (X \wedge Y) \wedge h^{-1}(Y) = R^1 \wedge Y.  
        \end{equation} 
        This completes the first step of the induction. 
        
        Assume that 
        \begin{equation*} 
		h^{-1}(T^{i}) \succeq R^{i-1} \wedge Y
	\end{equation*} 
	is true for  $i \geq 2$.  
	Notice that the right hand side is admissible for  $h$.  
	Since  $S = g^{-1}(X \wedge Y)$  is admissible for  $g$,  
	$T^{i+1} = (gh)(T^{i} \wedge S) = h(g(T^{i} \wedge S))$.   
	Moreover, 
	\begin{align*} 
	        h^{-1}(T^{i+1}) & = g(T^{i}  \wedge S) \\ 
			& \succeq g(h(R^{i-1} \wedge Y) \wedge S)  \quad \text{(by induction hypthesis)} \\ 
			&  \succeq g(h(R^{i-1})).  
	\end{align*} 
	On the other hand, 
	since  $T^{i+1} \succeq T^2$  by Lemma \ref{Lem:Filtration},  
	we have 
	\begin{equation*}
               h^{-1}(T^{i+1})  \succeq  h^{-1}(T^2)  \succeq R^1 \wedge Y
         \end{equation*}
         by (\ref{Eq:Y}).  
	Combining the above two, 
         	\begin{align*} 
	        h^{-1}(T^{i+1}) &  \succeq g(h(R^{i-1})) \wedge (R \wedge Y)  \\
	                         & =  (hg)(R^{i-1}) \wedge (hg)(D) \wedge Y \\
	                         & = (hg)(R^{i-1} \wedge D) \wedge Y = R^i \wedge Y,  
	         \end{align*} 
         and we have established the first claim. 
         
         The argument above is to analyze how the target block  $T^{i+1}$  of  $(gh)^{i+1}$  
         is mapped by  $h^{-1}$.  
         The second claim can be verified by replacing the role of  $gh$  by  $hg$.  
\end{proof}

\subsection{Elements of finite order}

From now on, 
we proceed our discussion under the assumption 
that the product  $gh$  and hence  $hg$  are torsion elements of order  $p$.  
The goal for the moment is to establish the diagram of bijections: 
\begin{equation*} 
	\begin{CD} 
		S^p @>{gh}>> T^p @= S^p \\ 
		@. @A{h}AA  @VV{g}V \\
		D^p @>>{hg}> R^p @= D^p 
	\end{CD}  
\end{equation*} 
Vertical arrows in the diagram make sense because of 
Remark \ref{Rm:Tadmissible} and \ref{Rm:Radmissible}.  

\begin{proposition}\label{Prop:TandR}  
	Assume that  $gh \in nV$  is a torsion element of order  $p$.  
	Then,  
	$T^p$  is admissible for  $g$  and so is  $R^p$  for  $h$,  
	and moreover,  
	\begin{equation*} 
		g(T^p) = R^p  \quad \text{and}  \quad  h(R^p) = T^p.   
	\end{equation*} 
\end{proposition} 
\begin{proof} 
	Since  $(gh)^p = id$, 
	$S^p = T^p$  and they are invariant under the action of  $gh$.            
	Also, 
	$T^p$  is admissible for both  $(gh)^{-1}$  and  $h^{-1}$.  
	Thus we have 
	$T^p = ((gh)^{-1})(T^p) = g^{-1}(h^{-1}(T^p))$, 
	or equivalently  $g(T^p) = h^{-1}(T^p)$.  
	In particular, 
	$T^p \, (= S^p)$  is admissible for  $g$  and  $h^{-1}$.  
	For further computation, 
	noticing that since  $S^i \succeq S^2$  by Lemma \ref{Lem:FiltrationII}, 
	we have   
	\begin{equation}\label{Eq:SucceqRelation} 
		g(S^i) \succeq g(S^2) = g((gh)^{-1}(T \wedge S)) \succeq h^{-1} (X \wedge Y) = D^1 
	\end{equation} 
	for all  $i \geq 2$.  
	Then, 
	by Lemma \ref{Lem:TandR} and (\ref{Eq:SucceqRelation}), 
	we obtain 
       	\begin{equation*} 
		g(S^p) = g(T^p) = h^{-1}(T^p) \succeq R^{p-1} \wedge D^1,  
	\end{equation*} 
	and hence, 
	applying  $hg$  on both sides of the last relation, 
	we have 
	\begin{equation*} 
		g(T^p) \succeq g(h(R^{p-1} \wedge D^1)) = R^p. 
	\end{equation*} 
	Replacing the role of  $T^p$  by  $R^p$  in the argument above  
	and noting that  $(hg)(R^p) = g(h(R^p)) = R^p$, 
	we also have 
	\begin{equation*} 
		g^{-1}(R^p) = h(R^p) \succeq T^{p}. 
	\end{equation*} 
	Combining these two formulae, 
	we have 
	\begin{equation*} 
		g(T^p)= R^p. 
	\end{equation*} 
	
	Applying the same argument for  $hg$  and we obtain   
	\begin{equation*} 
		h(R^p)= T^p. 
	\end{equation*} 
\end{proof} 

\begin{lemma}\label{Lem:Flip}
	If  $gh$  \emph{(} and hence  $hg$ \emph{)}  have order  $p$, 
	then 
	\begin{equation*} 
		g(R^p) = h^{-1}(R^p) 
		\quad \text{and} \quad 
		h(T^p) = g^{-1}(T^p).   
	\end{equation*} 	
\end{lemma} 
\begin{proof}  
	By the identity  $hg(hg)^{p-1} = id$  and Proposition \ref{Prop:TandR}, 
	we have 
	\begin{equation*} 
		R^p = h^{-1}(T^p) 
		= (g(hg)^{-1})(T^p) 
		= (hg)^{-1}(g(T^p)) 
		= (hg)^{-1}(R^p), 
	\end{equation*} 
	hence  $R^p$  is invariant under the action of  $hg$.  
	Since  $R^p$  is admissible for both  $g$  and  $h$, 
	the conclusion follows. 
	
	If we start with the identity  $gh(gh)^{p-1} = id$, 
	we can verify the second claim.  
\end{proof} 

\begin{proposition}\label{Prop:Identity} 
	Let  $p$  be the order of  $gh \in nV$.  
	Then  $T^p$  and  $R^p$  are invariant under the action of both  $g$  and  $h$. 
	Moreover, 
	\begin{equation*} 
		(D^p =) \; R^p = T^p \; (= S^p)   
	\end{equation*} 
	holds.  
\end{proposition} 
\begin{proof} 
	Since 
	\begin{equation*} 
		(hg)(R^p) = (hg)(D^p) = R^p = (hg)(R^{p-1} \wedge D^1),
	\end{equation*} 
	$R^p = R^{p-1} \wedge D^1$  and we have 
	\begin{equation*} 
		R^p 
		= R^{p-1} \wedge D^1 \succeq R^1 \wedge D^1 
		= g(X \wedge Y) \wedge h^{-1}(X \wedge Y) 
		\succeq X \wedge Y.  
	\end{equation*} 
	Since  $X \wedge Y$  is admissible for  $g$  and  $h^{-1}$, 
	\begin{align*} 
		g(R^p) & \succeq g(X \wedge Y) = R^1 \\
		h^{-1}(R^p) & \succeq h^{-1}(X \wedge Y) = D^1, 
	\end{align*} 
	and hence by Lemma \ref{Lem:Flip}
	\begin{equation}\label{Eq:g(R^p)}  
		g(R^p) = h^{-1}(R^p) \succeq R^1 \wedge D^1.  
	\end{equation} 
	Assume that 
	\begin{equation}\label{Eq:Hypothesis} 
		g(R^p) = h^{-1}(R^p) \succeq R^i \wedge D^1.  
	\end{equation} 
	is true for  $1 \leq i \leq p-1$.  
	Since  $R^i \wedge D^1$  is admissible for  $h$,  
	acting  $h$  on both sides of (\ref{Eq:Hypothesis}), 
	we have 
	\begin{equation*} 
		R^p \succeq h(R^i \wedge D^1).   
	\end{equation*}  
	Since 
	\begin{equation*} 
		h(R^i \wedge D^1) \succeq  h(D^1) = h(h^{-1}(X \wedge Y)) = X \wedge Y 
	 \end{equation*} 
	 and  $X \wedge Y$  is admissible for  $h$,  
	 $h(R^i \wedge D^1)$  and  $R^p$  are admissible for  $g$.  
	 Thus we have 
	 \begin{equation*} 
	 	g(R^p) \succeq g(h(R^i \wedge D^1)) 
		= (hg)(R^i \wedge D^1) = R^{i+1}.  
	\end{equation*}  
	Recall that  $g(R^p) = h^{-1}(R^p) \succeq D^1$, 
	and (\ref{Eq:Hypothesis}) is true when  $i = p-1$  by induction.  
	Thus we obtain 
	\begin{equation*} 
		g(R^p) = h^{-1}(R^p) \succeq R^{p-1} \wedge D^1 
		=(gh)^{-1}(R^p) = R^p.   
	\end{equation*} 
	Then  $g(R^p) = R^p$  because  $|g(R^p)| = |R^p|$.  
	
	The second identity follows from the same argument starting 
	from the identity 
	\begin{equation*} 
		(gh)(T^p) = (gh)(S^p) = T^p = (gh)(T^{p-1} \wedge S^1).  
	\end{equation*} 
	Also by Proposition \ref{Prop:TandR}, 
	\begin{equation*} 
		T^p = g(T^p) = R^p.  
	\end{equation*}  
\end{proof} 

\begin{remark} 
	The order  $p$  in the above argument can be replaced 
	by any multiple of  $p$.   
\end{remark} 

\begin{corollary}\label{Cor:TwoGenerator} 
 	Let  $g, h$  and  $gh$  be torsion elements in  $nV.$ 
	Then, 
	the group  $\langle g, \, h \rangle$  generated by  $g$  and  $h$  is a finite group.    
\end{corollary}  
\begin{proof}  
	Let  $p$  be the order of  $gh$.  
	Then,  
	$R^p$  is invariant under any word in  $g$  and  $h$  by Proposition \ref{Prop:Identity}, 
	and hence, 
	$\langle g, \, h \rangle \leq \mathfrak{S}_{|R^{p}|}$  and  
	$|\langle g, \, h \rangle| \leq |\mathfrak{S}_{|R^{p}|}| = |R^p|!$   
\end{proof} 

We want to generalize Corollary \ref{Cor:TwoGenerator} to 
a general case where we consider a finite generated group with any number of generators.
To see this, 
recall the definition of torsion locally finiteness.  

\begin{definition}[Torsion local finiteness]
         A group $G$ is said to be torsion locally finite 
         if all its finitely generated torsion subgroups are finite. 
\end{definition}

\begin{theorem}[Restatement of Theorem \ref{Thm:Main1} ]\label{Thm:TorsionLocallyFinite}  
 	$nV$  is torsion locally finite for  $n \geq 1$.  
\end{theorem} 
\begin{proof} 
	We want to prove this by induction on the number of 
	generators of a torsion subgroup in  $nV$.  
	Thus, 
	assume that 
	a torsion subgroup  $H= \langle h_1,h_2, \cdots, h_s \rangle < nV$  is finite,    
	and choose a torsion element  $g \not\in H$.  
	We want to show that if  $\langle g, H \rangle$  is torsion, 
	then it is finite.  
	Corollary \ref{Cor:TwoGenerator} establishes the claim for the case of  $s = 1$.  

	Since $\langle g, H \rangle$ is torsion, 
        $gh_1$ is of finite order, 
        and by Proposition \ref{Prop:Identity}, 
        there exists a dyadic block  $X_1$ 
        such that 
        \begin{equation*}
      		g(X_1) = h_1(X_1) = X_1.
        \end{equation*} 
        The hypothesis states   
        that there is a dyadic block  $Y$  which is invariant 
        under the action of any element of  $H$, i.e.,   
        \begin{equation*} 
		Y = h_1(Y) = h_2(Y) = \cdots = h_s(Y).   
	\end{equation*} 
 
 	Consider the torsion element $gh_1h_2 \in \langle g, H \rangle$.  
	Notice that both  $X_1$  and  $Y$  are invariant under the action of  $h_1$.     
	The product  $g(h_1h_2)$  of  $g$  and  $h_1 h_2$    
	is represented as a multiplication of dyadic blocks by 
        \begin{align*} 
	          (g^{-1}(X_1), X_1) \cdot (Y, (h_1h_2)(Y)) 
	                       &  = (g^{-1}(X_1 \wedge Y), (h_1h_2)(X_1 \wedge Y)) \\
	                       &  = (g^{-1}(X_1 \wedge Y), h_2(h_1(X_1 \wedge Y))) \\
	                       & = (g^{-1}(X_1 \wedge Y), h_2(X_1 \wedge Y)).        
        \end{align*}  
        By the proof of Proposition \ref{Prop:Identity}, 
        we have a dyadic block $X_2$  determined by  
        $(g^{-1}(X_1 \wedge Y), h_2(X_1 \wedge Y))$  and the order of  
        $gh_1h_2$  
        so that  we have
        \begin{equation}\label{Eq:g(X_2)=X_2} 
		g(X_2) = X_2 = (h_1h_2)(X_2).   
	\end{equation}  
        
	On the other hand, 
	regarding  $g h_1 h_2$  as a product of  $gh_1$  and  $h_2$  
	and we represent it as a multiplication of block pairs,
	\begin{align*} 
		((gh_1)^{-1}(X_1), X_1) \cdot (Y, h_2(Y)) 
		&  = ((gh_1)^{-1}(X_1 \wedge Y), h_2(X_1 \wedge Y)) \\
		&  = (g^{-1}(h_1^{-1}(X_1 \wedge Y)), h_2(X_1 \wedge Y)) \\
		& = (g^{-1}(X_1 \wedge Y), h_2(X_1 \wedge Y)).  
        \end{align*} 
        This is exactly the same dyadic block pair that represents  $g(h_1h_2)$  and 
        since the order of  $(gh_1)h_2$  is equal to that of  $g(h_1h_2)$, 
        the dyadic block obtained by the method of the proof of Proposition \ref{Prop:Identity}  
        is the same one, $X_2$,  and  
	\begin{equation}\label{Eq:gh1(X_2)=X_2} 
		(gh_1)(X_2) = X_2 = h_2(X_2).  
	\end{equation} 
	Then from (\ref{Eq:g(X_2)=X_2}) and (\ref{Eq:gh1(X_2)=X_2}), 
	we have 
	\begin{equation*} 
		h_1(X_2) = h_1(g(X_2)) = (gh_1)(X_2) = X_2.  
         \end{equation*} 
         In particular, 
         $X_2$  is invariant under the action of  $g, h_1$  and  $h_2$.  

	We now apply the same argument to the products  $(gh_1h_2)h_3$  and  
	$(gh_1)(h_2 h_3)$, 
	and obtain a dyadic block  $X_3$  which is invariant under 
	the action of  $g, h_1, h_2$  and $h_3$.  
	Repeating the argument to the product  $(gh_1 \cdots h_i)(h_{i+1} \cdots h_s)$  
	and we finally obtain that  
	there is a dyadic block  $X_s$  such that $X_s$ is invariant under 
	the action of  $g, h_1, h_2, \cdots, h_s$,  
	and in particular,  $\langle g, H \rangle$  is finite. 
\end{proof}  

\begin{remark}  
	Theorem \ref{Thm:TorsionLocallyFinite} was proven for the $n = 1$ case by  \cite{BCST}, 
	and our proof provides an alternative one in this case.  
\end{remark} 
 
%
%
\medskip 
\section{Embedding $\mathbb{Q}$ into $2V$}
\label{Sec:Roots} 

We now turn to study the elements of infinite order in  $nV$  
and prove Theorem \ref{Thm:Main2}  in the introduction.
We first recall the definition of roots,   

\begin{definition}
        For a group $G,$ 
        an element $h \in G$ is said to be a root of $g \in G$ 
        if there is  $k \geq 2$  such that $h^k = g$,  
        and the smallest such $k$  is called the order of the root $h$. 
\end{definition} 


It is known by G. Higman  \cite{Higman}  that, 
Thompson group  $V$  does not have an element of infinite order 
with roots of arbitrary large order, 
thus fails to contain  $\mathbb{Q}$.
On the other hand, 
the following example motivated us to show that 
it is not the case for  $n \geq 2$. 

\begin{example}\label{Ex:Root} 
	We start with an infinite order element  $h$  
	as depicted on the 
	right
	of Figure \ref{Fig:Root}
	which is represented by  $(S, T)$ with a labeling in each subblock.  
	The induced vertical permutation by  $h$  is the identity, 
	namely the subblock labelled by  
	$*$
	in  $S$  is mapped to the subblock 
	$*$
	in   $T$.     
	A root  $h_1$  of  $h$  of order  $2$  is 
        depicted on the 
        right
        of Figure \ref{Fig:Root} as  $(S_1, T_1)$.   
	\begin{figure}[htp]
		\begin{center} 	      
		\begin{tikzpicture}[baseline=-0.65ex, thick, scale=0.3]
                 \draw (0, -5) to (0,5);
                 \draw (0, -5) to (10, -5);
                 \draw (10, -5) to (10,5);
                 \draw (10, 5) to (0, 5);
                 \draw (5, -5) to (5,5);
                 \draw (5, -2.5) to (10,-2.5);
                 \draw (5, 0) to (10,0);
                 \draw (2.5, 0) node{1};
                 \draw (7.5, 2.5) node{a};
                 \draw (7.5, -1.25) node{b};
                 \draw (7.5, -3.75) node{c};
         \end{tikzpicture}
         \begin{tikzpicture}[baseline=-0.65ex, thick, scale=0.3]
                 \draw (5, 0) node{$\mapsto$};
                 \draw (5, -7) node{$h_1$};
                 \end{tikzpicture}
         \begin{tikzpicture}[baseline=-0.65ex, thick, scale=0.3]
                 \draw (0, -5) to (0,5);
                 \draw (0, -5) to (10, -5);
                 \draw (10, -5) to (10,5);
                 \draw (10, 5) to (0, 5);
                 \draw (5, -5) to (5,5);
                 \draw (5, -5) to (5,5);
                 \draw (0, 2.5) to (5,2.5);
                 \draw (0, 0) to (5,0);
                 \draw (2.5, 3.75) node{a};
                 \draw (2.5, 1.25) node{b};
                 \draw (2.5, -2.5) node{c};
                 \draw (7.5, 0) node{1};
        \end{tikzpicture}
                 \quad
        \begin{tikzpicture}[baseline=-0.65ex, thick, scale=0.3]
                 \draw (0, -5) to (0,5);
                 \draw (0, -5) to (10, -5);
                 \draw (10, -5) to (10,5);
                 \draw (10, 5) to (0, 5);
                 \draw (0, 0) to (10,0);
                 \draw (0,-2.5) to (10, -2.5);
                 \draw (5, 2.5) node{a};
                 \draw (5, -1.25) node{b};
                 \draw (5, -3.75) node{c};
        \end{tikzpicture}
        \begin{tikzpicture}[baseline=-0.65ex, thick, scale=0.3]
                 \draw (5, -7) node{$h$};
                 \draw (5, 0) node{$\mapsto$};
                 \end{tikzpicture}	
        \begin{tikzpicture}[baseline=-0.65ex, thick, scale=0.3]
                 \draw (0, -5) to (0,5);
                 \draw (0, -5) to (10, -5);
                 \draw (10, -5) to (10,5);
                 \draw (10, 5) to (0, 5);
                 \draw (0, 0) to (10,0);
                 \draw (0,2.5) to (10, 2.5);
                 \draw (5, 3.75) node{a};
                 \draw (5, 1.25) node{b};
                 \draw (5, -2.5) node{c};               
        \end{tikzpicture}
                 	\caption{$h$  and its root  $h_1$  of order  $2$}
			\label{Fig:Root}
		\end{center} 	
	\end{figure} 
\end{example} 
        
Divide the dyadic block with a single member vertically into two subblocks
and squeeze $S$ to the right half in order to obtain  $S_1$  and similarly, 
we divide the single dyadic block vertically into two subblocks 
and squeeze $T$ to the left half to obtain  $T_1$.  
Giving labelings by combining that of  $h$  and 
a rotation of  order  $2$  in the horizontal direction, 
we obtain the resulted block pair  $(S_1, T_1)$.

We now define the following sequence of dyadic blocks 
in order to construct a sequence of roots that we need later on.

\begin{definition}[$L_k$]
        We take a dyadic block with a single member 
        and divide it into two vertical parts, 
        then we obtain $L_2$.
        Then we divid the left half of $L_2$ again vertically into half and we obtain $L_3$.
        By repeating this process we obtain a dyadic block with $k$ vertical subblocks 
        and denoted it by $L_k$ (see Figure \ref{Fig:block}).
\begin{figure}[htp]
           \begin{center} 
             \begin{tikzpicture}[baseline=-0.65ex, thick, scale=0.3]
                 \draw (0, -5) to (0,5);
                 \draw (0, -5) to (10, -5);
                 \draw (10, -5) to (10,5);
                 \draw (10, 5) to (0, 5);               
                 \draw (5, -5) to (5,5);     
                                  \draw (5, -7) node{$L_2$};                               
                 \end{tikzpicture}\quad
                 \begin{tikzpicture}[baseline=-0.65ex, thick, scale=0.3]
                 \draw (0, -5) to (0,5);
                 \draw (0, -5) to (10, -5);
                 \draw (10, -5) to (10,5);
                 \draw (10, 5) to (0, 5);               
                 \draw (5, -5) to (5,5);
                 \draw (2.5, -5) to (2.5,5);         
                 \draw (5, -7) node{$L_3$};                                     
                 \end{tikzpicture}\quad	
                 \begin{tikzpicture}[baseline=-0.65ex, thick, scale=0.3]
                 \draw (0, -5) to (0,5);
                 \draw (0, -5) to (10, -5);
                 \draw (10, -5) to (10,5);
                 \draw (10, 5) to (0, 5);               
                 \draw (5, -5) to (5,5);
                \draw (1.25, -5) to (1.25,5);
                 \draw (2.5, -5) to (2.5,5);            
                 \draw (5, -7) node{$L_4$};                                            
                 \end{tikzpicture}\quad
                 \begin{tikzpicture}[baseline=-0.65ex, thick, scale=0.3]
                 \draw (0, -5) to (0,5);
                 \draw (0, -5) to (10, -5);
                 \draw (10, -5) to (10,5);
                 \draw (10, 5) to (0, 5);               
                 \draw (5, -5) to (5,5);
                 \draw (0.625, -5) to (0.625,5); 
                \draw (1.25, -5) to (1.25,5);
                 \draw (2.5, -5) to (2.5,5);            
                 \draw (5, -7) node{$L_5$};                                            
                 \end{tikzpicture}
                   \begin{tikzpicture}[baseline=-0.65ex, thick, scale=0.3]
                 \draw (5, 0) node{$\cdots$};  
                 \end{tikzpicture}
         	\caption{Dyadic block $L_2$, $L_3$, $L_4$, $L_5$ }
			\label{Fig:block}
		\end{center} 	
\end{figure} 
\end{definition}

\begin{proposition}\label{Prop:QEmbeds2V}
       The group $nV$  contains a subgroup isomorphic to  $\mathbb{Q}$ for $n \geq 2 $.
\end{proposition}
\begin{proof} 
         Since there is an obvious sequence of isomorphic embeddings of 
         Brin-Thompson groups  $2V < 3V < \cdots$, 
         we only need to prove Theorem  \ref{Thm:Main2}  for  $n = 2$. 
 	Next, we construct a sequence of groups in $2V$
	\begin{equation}\label{Eq:seq} 
		\langle s_1 \rangle  <  \langle s_2 \rangle  < \cdots <  \langle s_k \rangle 
	\end{equation}
	generated by roots of  $h$  of order $k!$.
	Here we take $s_1 = h$ and $s_2 = h_1$ 
	as constructed in Figure \ref{Fig:Root}.
	We then construct $s_3$ as follows.   
	Start with a block in the form of  $L_2$, 
	then we embed  $L_3$  into each subblock of $L_2$  with horizontal scaling 
	and we denote it by $L_{2,3}$.  
	The product of the numbers in the tuples in the suffix 
	indicates the number of vertical blocks  $3\cdot 2 = 3!$.  
	Then, 
	we let  $s_3$  be generate by a horizontal slide with  
	a ``vertical shifting" at the central vertical line
	as depicted in Figure \ref{Fig:jthRoot}.
 	\begin{figure}[htp]
		\begin{center}
          \begin{tikzpicture}[baseline=-0.65ex, thick, scale=0.3]	
		  \draw (0, -5) to (0,5);
                 \draw (0, -5) to (10, -5);
                 \draw (10, -5) to (10,5);
                 \draw (10, 5) to (0, 5);               
                 \draw (5, -5) to (5,5);
                 \draw (1.25, -5) to (1.25,5); 
                 \draw (2.5, -5) to (2.5,5); 
                 \draw (5, -2.5) to (6.125,-2.5);
                 \draw (5, 0) to (6.125,0);             
                 \draw (6.125, -5) to (6.125,5);    
                 \draw (7.5, -5) to (7.5,5);        
                   
                 \draw (5.65, 2.5) node{a};
                 \draw (5.65, -1.25) node{b};
                 \draw (5.65, -3.75) node{c};   
                 
                 \draw (0.75, 0) node{3};
                 \draw (1.9, 0) node{4};
                 \draw (3.75, 0) node{5};   
                 \draw (6.9, 0) node{1};   
                 \draw (8.75, 0) node{2};   
           \end{tikzpicture}
           \begin{tikzpicture}[baseline=-0.65ex, thick, scale=0.3]
                 \draw (5, -7) node{$s_{3}$};
                 \draw (5, 0) node{$\mapsto$};
                 \end{tikzpicture}	
                 \begin{tikzpicture}[baseline=-0.65ex, thick, scale=0.3]
               \draw (0, -5) to (0,5);
                 \draw (0, -5) to (10, -5);
                 \draw (10, -5) to (10,5);
                 \draw (10, 5) to (0, 5);               
                 \draw (5, -5) to (5,5);
                 \draw (1.25, -5) to (1.25,5); 
                 \draw (2.5, -5) to (2.5,5);          
                 \draw (6.125, -5) to (6.125,5);    
                 \draw (7.5, -5) to (7.5,5);      
               
                 \draw (2.5, 0) to (5,0);
                 \draw (2.5,2.5) to (5, 2.5);
                 \draw (3.75, 3.75) node{a};
                 \draw (3.75, 1.25) node{b};
                 \draw (3.75, -2.5) node{c};      
                  
                  \draw (0.75, 0) node{4};
                 \draw (1.9, 0) node{5};
                 \draw (5.75, 0) node{1};   
                 \draw (6.9, 0) node{2};   
                 \draw (8.75, 0) node{3};        
                 \end{tikzpicture}
			\caption{The root  $s_3$  of  $s_2$  of order  $3$}
			\label{Fig:jthRoot}
		\end{center} 	
    	\end{figure} 
	
	Finally,  
	by repeating the construction of the block pairs $L_{1,2, \cdots, k}$,  
	we obtain a sequence of elements  $s_k \in 2V$ 
	so that each generates a finite cyclic group 
	which contains the ones generated by the previous elements in the sequence
	and in particular  $\langle s_k \rangle$  can be embedded in  $\mathbb{Q}$.
	Thus, 
	we have the following inductive limit 
	\begin{equation*} 
		\bigcup_{k = 1}^{\infty} \, \langle s_k \rangle 
	\end{equation*}
 	that contains a root of  $h$  of arbitrary order and hence    
	is isomorphic to  $\mathbb{Q}$.  
\end{proof}

\begin{theorem}[Restatement of Theorem \ref{Thm:Main2}]\label{Thm:EmbeddingQ} 
	If  $n \geq 2$,  
	then the Brin-Thompson group  $nV$  contains continuum many copies of
	the additive group of the rationals  $\mathbb{Q}$  sharing 
	the subgroup isomorphic to  $\mathbb{Z}$.  
\end{theorem}  

To introduce a variant of the above construction,   
	we introduce $R_k$ similarly as follows.
	\begin{definition}[$R_k$] 
	        Take a dyadic block with a single member 
		and divide it into two vertical parts, 
		and we obtain  $R_2 = L_2$.
		Then we divid the right half of  $R_2$  again vertically into half and we obtain  $R_3$.
		By repeating this process we obtain a dyadic block with $k$ vertical subblocks 
		and denote it by  $R_k$ (see Figure \ref{Fig:R_k}).
	\end{definition}
 	\begin{figure}[htp]
		 \begin{center} 	
	         \begin{tikzpicture}[baseline=-0.65ex, thick, scale=0.3]	
	         \draw (0, -5) to (0,5);
                 \draw (0, -5) to (10, -5);
                 \draw (10, -5) to (10,5);
                 \draw (10, 5) to (0, 5);               
                 \draw (5, -5) to (5,5); 
                 \draw (5, -7) node{$R_2 = L_2$};                                                  
                 \end{tikzpicture} \quad	
                    \begin{tikzpicture}[baseline=-0.65ex, thick, scale=0.3]	
	         \draw (0, -5) to (0,5);
                 \draw (0, -5) to (10, -5);
                 \draw (10, -5) to (10,5);
                 \draw (10, 5) to (0, 5);               
                 \draw (5, -5) to (5,5);
                 \draw (7.5, -5) to (7.5,5);      
                 \draw (5, -7) node{$R_3$};                                                              
                 \end{tikzpicture} \quad
                    \begin{tikzpicture}[baseline=-0.65ex, thick, scale=0.3]	
	          \draw (0, -5) to (0,5);
                 \draw (0, -5) to (10, -5);
                 \draw (10, -5) to (10,5);
                 \draw (10, 5) to (0, 5);               
                 \draw (5, -5) to (5,5);
                 \draw (7.5, -5) to (7.5,5);      
                 \draw (8.725, -5) to (8.725,5);    
                 \draw (5, -7) node{$R_4$};                  
                 \end{tikzpicture} \quad
                    \begin{tikzpicture}[baseline=-0.65ex, thick, scale=0.3]	
	      \draw (0, -5) to (0,5);
                 \draw (0, -5) to (10, -5);
                 \draw (10, -5) to (10,5);
                 \draw (10, 5) to (0, 5);               
                 \draw (5, -5) to (5,5);
                 \draw (7.5, -5) to (7.5,5);      
                 \draw (8.725, -5) to (8.725,5);    
                 \draw (9.35, -5) to (9.35,5);      
                 \draw (5, -7) node{$R_5$};              
                 \end{tikzpicture}
                   \begin{tikzpicture}[baseline=-0.65ex, thick, scale=0.3]
                 \draw (5, 0) node{$\cdots$};  
                 \end{tikzpicture}	
			\caption{Dyadic block  $R_k$}
			\label{Fig:R_k}
		\end{center} 	
    	\end{figure} 

\begin{proof}[Proof of Theorem \ref{Thm:EmbeddingQ}] 	
	To each  $k \geq 2$, 
	we have a choice to choose either  $L_k$  or  $R_k$.  
	It is easy to check that a different choice yield a different $\mathbb{Q}$.  
	The number of choices can be counted as a selection function 
	from  $\mathbb{N}_{\geq 2}$  to  $\{ L, \, R \}$  which has continuum cardinalily.  
\end{proof}

This result on the nature of infinite order elements in a group also have some indications 
on the geometric properties of the groups.  
Farley constructed in  \cite{F}  a ${\rm CAT}(0)$ cube complex
such that the group $V$ acts on properly. 
However, 
it is no longer true for  $n \geq 2$   by  Callard and Salo  \cite{CS}.  
We could provide another argument to show the same result,  

\begin{corollary}
       The group  $nV$  does not act on  ${\rm CAT}(0)$  
       cube complexes properly by isometries
       for  $n \geq 2$.  
\end{corollary}
\begin{proof} 
	The isometries of a ${\rm CAT}(0)$ space with unbounded orbits act 
	as translations on some bi-infinite geodesics due to Haglund  \cite{Haglund}, 
	hence the order of the roots of an infinite order element 
	can not exceed the translation length.
\end{proof}

%
%



%
%


\medskip 
\begin{thebibliography}{99} 

\bibitem{BHM} 
James Belk, James Hyde and Francesco Matucci, 
Embedding  $\mathbb{Q}$  into a finitely presented group, 
Bill. Amer. Math. Soc., 59 (2022), 561-567.  

\bibitem{BZ}
James Belk and Matthew C. B. Zaremsky,
Twisted Brin-Thompson groups, 
Geom. Topol., 26 (2022), 1189-1223.

\bibitem{Brin}
Matthew G. Brin, Higher Dimensional Thompson Groups, 
Geometriae Dedicata 108 (2004), 164-192.  

\bibitem{BCST}
Jos\'{e} Burillo, Sean Cleary, Melanie Stein and Jennifer Taback, 
Combinatorial and metric properties of Thompson's group  $T$, 
Trans. Amer. Math. Soc., 353 (2001), 1677-1689.  

\bibitem{BCR}
Jos\'{e} Burillo, 
Sean Cleary and Claas E R\"{o}ver, 
Obstructions for subgroups of Thompson's group $V$, 
London Math. Soc. Lecture Note Ser. vol 444 (2018), 1-4.  

\bibitem{CFP}
James W. Cannon,  William J. Floyd, Walter R. Parry, 
Introductory Notes on Richard Thompson's groups, 
l'Enseignement Math\'ematique, 42 (1996), 215-256.  

\bibitem{CS} 
Antonin Callard, Ville Salo,
Distortion element in the automorphism group of a full shift, 
Ergod. Th. Dynam. Sys., 
44 (2024), 1757-1817. 

\bibitem{F} 
Daniel. S. Farley,
Finiteness and ${\rm CAT(0)}$ properties of diagram groups,
Topology, 42 (2003), 1065-1082. 

\bibitem{Grigorchuk} 
Rostislav Grigorchuk,
On Burnside's problem on periodic groups, 
(Russian) Funktsionalyi Analiz i ego Prilozheniya, 14 (1980), 53-54.  

\bibitem{Higman} 
Graham Higman, 
Finitely presented infinite simple groups, 
Australian National University, 
Notes on Pure Mathematics, 8 (1974).  

\bibitem{Haglund} 
Fr\'ed\'eric Haglund, 
Isometries of CAT(0) cube complexes are semi-simple, 
Ann. Math, Qu\'ebec. 47 (2023), 249--261. 
%
\bibitem{K}
V. D. Mazurov and E. I. Khukhro, 
The Kourovka notebook: Unsolved Problems in Group Theory, 
the Kourovka Notebook, arXiv: 1401.0300v39, (2026).

\bibitem{R}
Claas E, R\"{o}ver, 
Subgroups of finitely Presented Simple Groups,
PhD Thesis, University of Oxford, UK (1999).

\bibitem{Rover}
Claas E, R\"{o}ver,  
Constructing finitely presented simple groups 
that contain Grigorchuk groups, 
J. Algebra, 220 (1999), 284-313.    

\bibitem{Sheng} 
Xiaobing Sheng, 
Some obstructions on the subgroups of the Brin-Thompson groups and a selection of 
twisted Brin-Thompson groups, 
Unpublished manuscript  (2022). 

\end{thebibliography}
\end{document}